\documentclass[a4paper,11pt]{amsart}
\usepackage{hyperref,latexsym}
\usepackage{enumerate}

\theoremstyle{plain}
\newtheorem{theorem}{Theorem}[section]
\newtheorem{lemma}[theorem]{Lemma}
\newtheorem{corollary}[theorem]{Corollary}

\theoremstyle{definition}

\theoremstyle{remark}
\newtheorem{remark}{Remark}

\newcommand{\arctanh}{\rm arctanh}

\begin{document}

\title[Rigidity for convex billiards]
      {Hopf rigidity for convex billiards on the Hemisphere and Hyperbolic Plane}

\date{9 May 2012}
\author{Misha Bialy}
\address{School of Mathematical Sciences, Raymond and Beverly Sackler Faculty of Exact Sciences, Tel Aviv University,
Israel} \email{bialy@post.tau.ac.il}
\thanks{Partially supported by ISF grant 128/10}

\subjclass[2000]{ } 

\begin{abstract}
This paper deals with Hopf type rigidity for convex billiards on
surfaces of constant curvature. We prove that the only convex
billiard without conjugate points on the Hyperbolic plane or on the
Hemisphere is circular billiard.

\end{abstract}

\maketitle

\section{Introduction and the result}In this paper we consider
convex billiard with smooth (class $C^2$) boundary curve $\gamma$
which lies on a constant curvature $K=\pm1$ surface, denoted by $S$,
hemisphere or hyperbolic plane. We shall assume that geodesic
curvature $k$ of the boundary curve $\gamma$ everywhere positive.

Inspired by a famous theorem of E.Hopf \cite {H} (see \cite {BI} for
the generalization), I proved for the Euclidean plane in \cite{B1}
that the only billiard which has no conjugate points (as a discrete
dynamical system) is circular billiard. This is in fact very
geometric result because as corollary one gets that the only
\textit{totally} integrable billiard in the plane is circular
billiard  (total integrability, after \cite{knauf}, means that the
phase cylinder is foliated by closed invariant curves which are not
null homotopic). The proof of this result consists of two steps
(very much like in the Hopf theorem): the first is the construction
of non-vanishing Jacobi field along every orbit, and the second is
an argument of integral geometry.

It was observed by M. Wojtkowski in \cite{W} that this result can be
proved by a different integral geometric argument which uses the
"mirror" formula of geometric optics. We refer to \cite {T}, \cite
{G} for the discussion and open problems and to \cite{I} to other
developments. However it was an old question which people tried to
answer if this Hopf rigidity still holds for other models of
billiards, which also give rise to twist symplectic maps of the
cylinder. For all of these models the first part of the proof
constructing non-vanishing Jacobi fields can be also performed. But
integral geometric part is not known and not clear.

In this paper we prove the  rigidity result for convex billiards on
surfaces of constant, positive or negative curvature. The idea is
again integral geometric and uses the "mirror formula" in the case
of constant curvature. Using this tool we reduce the claim to a
geometric inequality which turns out to be a consequence of the
isoperimetric inequality on the surface. However, in the case of
constant non-zero curvature this reduction and the reduced geometric
inequality become much more involved, just because the isoperimetric
inequalities on surfaces are essentially more complicated.
Remarkably in the case of sphere the reduced geometric inequality
can be obtained from Fenchel inequality for space closed curves, but
I am not aware of such a proof for Hyperbolic case. For both cases
we give the intrinsic proofs in Sections 5,6 below.

Let me remind the definition of conjugate points for billiard
configurations (see \cite{B1}) or more generally for a twist map.
Throughout this paper we choose an arc-length parameter $x$ on of
$\gamma$ and identify $x$ with the $\gamma(x)$ everywhere below.
Recall that billiard configurations $\{x_n\}$ are extremals of the
action functional
\begin{equation}
A\{x_n\}=\sum_{-\infty}^{\infty} L(x_n,x_{n+1}),\label{a}
\end{equation}
 where
$L(x,y)=dist_S(x,y)$ is the generating function of the billiard ball
map:
$$T:(x, \Phi)\rightarrow (y,\Psi),\
\Phi=-L_1(x,y),\ \Psi=L_2(x,y).$$ Here $\ \Phi=\cos\phi,\
\Psi=\cos\psi\ $ and $\ \phi,\psi\ $ are the angles formed by the
oriented geodesic segment $g_{xy}$ joining $x$ and $y$ with the
oriented boundary curve $\gamma$. Subindexes of $L$ here and later
stand for partial derivatives. Then it follows that twist condition
(see below) is satisfied and $T$ is a symplectic twist map of the
cylinder $\Omega=\gamma\times(-1,1)$. A Jacobi field along
configuration ${x_n}$ is a sequence ${\xi_n}$ satisfying the
discrete Jacobi equation:
\begin{equation}
\label{Jacobi}
 b_{n-1}\xi_{n-1}+a_n
\xi_n+b_{n}\xi_{n+1}=0,
\end{equation}
$$a_n=L_{22}(x_{n-1},x_n)+L_{11}(x_n,x_{n+1}),\
b_n=L_{12}(x_n,x_{n+1}).$$

Two points $x_M,x_N$ of the configuration are called conjugate if
the exists a non-trivial Jacobi field vanishing at these points.

Our main result is as follows:
\begin {theorem}
Let $\gamma$ be a smooth convex simple closed curve having positive
geodesic curvature lying on the hemisphere or on the hyperbolic
plane of constant curvature $\pm1$. If every billiard configuration
$\{x_n\}$ has no conjugate points, then $\gamma$ is a geodesic
circle.
\end{theorem}
\begin{corollary}
\label{cor} Suppose the billiard ball map $T$ is totally integrable,
meaning that through every point of the phase cylinder passes a
continuous non-contractible curve invariant under $T$. Then $\gamma$
is a geodesic circle.
\end{corollary}

Three remarks are in order:

1. The model of billiards on constant curvature surfaces (and even
more general models) is extensively studied. We refer to
\cite{Gutkin}, \cite{Zh} , \cite{GT},\cite{V} which are most
relevant to our result.

It is worth mentioning that similarly to the planar case the
elliptic domains on the hemisphere and hyperbolic plane are the only
known examples of integrable billiards (\cite{V}). Corollary \ref
{cor} proves a part of the Birkhoff conjecture on constant curvature
surfaces saying that the only integrable convex billiards are
ellipses. We refer to \cite {K} for other developments in the
direction of this conjecture.

2. In \cite{W} Wojtkowski uses the assumption of existence of the
so-called monotone invariant sub-bundle replacing the assumption of
no conjugate points. It is a consequence of \cite{B1} that these two
assumptions are in fact equivalent as we shall explain in Section 3.
However I prefer my original condition of no conjugate points.

3. It is a remarkable fact that for billiard in any convex domain of
Euclidean space of dimensions greater than 2 there always exist
conjugate points (see \cite {B2}). This is in fact a very general
phenomenon valid for Riemannian  billiards inside simple convex
domains, starting from dimension 3. It is still unclear to me how
billiard inside the round ball can be distinguished in terms of
variational properties of billiard orbits.

This paper is organized as follows. I shall discuss equivalent forms
of the "no conjugate points" condition for twist maps in Section 3
below. The proof of the main theorem uses "mirror" equation. We show
how one gets the "mirror" equation from our assumption of no
conjugate points in Section 4. The main  integral geometric part of
the proof is done separately: first for the hyperbolic plane in
Section 5 and then for the hemisphere in Section 6.

\section*{Acknowledgements} For many years we were playing mathematical billiards
with Sergei Tabachnikov  and it is a pleasure to thank him for many
discussions. He encouraged me to use the mirror formula instead of
my original approach. Thanks also to Yaron Ostrover and Daniel Rosen
who revived my interest to this subject again.
\section{Mirror formula}
The mirror formula of geometric optics for a convex mirror in the
Euclidean plane reads
$$ \frac {1}{a}+\frac {1}{b}=\frac{2k(x)}{\sin\phi}.$$
Here $x$ is a point on the mirror, $a$ is a distance from a point
$A$ inside the domain to the point $x$ along the shortest ray and
$b$ is a distance from $x$ along the reflected ray to the point $B$
where the focusing of the reflected beam occurs, $\phi$ is the angle
of reflection. The mirror formula can be generalized to simple
convex domains in Riemannian surfaces in a straightforward way.
(Being simple means that all geodesics inside the domain are
minimal, intersecting the boundary curve, in particular there are no
conjugate points along geodesics inside the domain). In Riemannian
case the mirror formula is the following:

$$\frac {Y_A^{'}(a)}{Y_A(a)}+\frac{Y_B^{'}(b)}{Y_B(b)}=\frac{2k(x)}{\sin\phi}.$$
Here the $Y$ denotes the orthogonal Jacobi field along geodesic
segments $g_{A,x}$ and $g_{B,x}$ going to $x$ from $A$ and $B$
respectively. In this paper we shall assume everywhere that subindex
$A$ at $Y$ means that $A$ is the initial point on geodesic segment
the initial conditions for Jacobi equation are.
$${Y_A(0)}=0; \
{Y_A^{'}(0)=1}.$$

It is important that for the case of constant curvature the Jacobi
field with these initial conditions do not depend on the point and
the direction and are equal to
$$Y(t)=\sinh t; \ Y(t)=\sin t,$$
for the case of hyperbolic plane and the hemisphere respectively. So
in this case the formula looks particularly simple (we refer to
\cite{Gutkin}, \cite{Zh} for some billiard applications of this
formula in constant curvature case, and to \cite{GT} for mirror
formula for other billiards):

$$\frac {Y^{'}}{Y}(a)+\frac{Y^{'}}{Y}(b)=\frac{2k(x)}{\sin\phi}.$$

 The proof of the mirror formula in general Riemannian case
follows immediately from the explicit computation of the second
derivatives of the length function $L$ (the calculations using Gauss
Bonnet theorem are easy
 and will not be reproduced here):

\begin{eqnarray}
\label{derivatives}
L_{11}(x,y)&=&\frac{Y_y^{'}}{Y_y}(L)\sin^2\phi-k(x)\sin\phi,\nonumber \\
L_{22}(x,y)&=&\frac{Y_x^{'}}{Y_x}(L)\sin^2\psi-k(y)\sin\psi,   \\
L_{12}(x,y)&=&\frac{\sin\phi\sin\psi}{Y_x(L)}=\frac{\sin\phi\sin\psi}{Y_y(L)}.\nonumber
\end{eqnarray}
Here $Y_x$ and $Y_y$ are Jacobi fields along the segments $g_{xy}$
and $g_{yx}$ with initial conditions at the points $x$ and $y$
respectively. Again in the case of constant curvature the subindexes
can be dropped.

\section{The condition of no conjugate points}
In this section $T$ denotes symplectic negative twist map of the
cylinder (not necessarily billiard map), with the generating
function $L$. Let me remind the condition used in \cite{W} instead
of no conjugate points condition.

Let $l$ be a \textit{non-vertical} sub-bundle of tangent lines to
the cylinder $\Omega$. In other words the line $l(x,\Phi)\subset
T_{(x,\Phi)} \Omega$ is assumed to be non-vertical for all
$(x,\Phi)\in\Omega$. Such a sub-bundle is called monotone in
\cite{W} if in addition it is equipped with the orientation chosen
on the lines by the condition $dx>0$.

In order to state all the equivalent conditions, I shall use also
the definition of locally maximizing configurations of negative
twist maps (note the positive sign of mixed derivative of $L$ in the
case of billiards). Infinite extremal $\{x_n\}$ is called locally
maximizing if any finite segment of the extremal is a local maximum
of the functional with the fixed end points.
\begin{theorem}
\label{equivalent}
 Given a twist area preserving map $T$ of the
cylinder $\Omega$ with the generating function $L$ (with the
positive cross derivative $L_{12}$, the so called negative twist).
The following conditions are equivalent:
\begin{enumerate}[(a)]

\item All configurations have no conjugate points;

\item All configurations are locally maximizing;

\item There exists a measurable positive function
$\nu:\Omega\rightarrow\mathbf{R}_+$ such that the cocycle
$$\nu_n(x,\Phi)=\nu(x,\Phi)\cdot...\cdot\nu(T^{n-1}(x,\Phi))$$
satisfies the discrete Jacobi equation (\ref{Jacobi}).

\item There exists measurable non-vertical, monotone sub-bundle defined a.e. on $\Omega$
which is invariant under the twist map $T$.

\end{enumerate}
\end{theorem}

It is a well known fact in Aubry-Mather theory that any
configuration $\{x_n\}$ which comes from an orbit lying on a closed
continuous non-contractible invariant curve of $T$ is locally
maximizing. This fact implies that if one assumes the "foliation
condition", meaning that there exists an invariant curve passing
through every point of the cylinder, then each of the conditions
(a)-(d) of Theorem \ref{equivalent} follows. This remark explains
Corollary \ref{cor} of the main theorem.

Let me remark also that it is plausible that this "foliation
condition" is in fact equivalent to the conditions (a)-(d), like it
is established (\cite{J} see also \cite{A}) for the case of the
flows. However, I don't know the proof of this statement for the
case of twist maps.
\begin{proof}
Condition (b) is equivalent to (a) for the following reason. The
matrix of the second variation of any locally maximizing segment is
negative semi-definite. Then it follows from \cite {MMS} that in
such a case it is in fact negative definite. The condition of "no
conjugate points" along one configuration is equivalent to the fact
that for any segment the matrix of second variation is
non-degenerate. Therefore there are no conjugate points along every
locally maximizing orbit, so (b) implies (a). In the opposite
direction, if all configurations are known to have no conjugate
points then second variation matrix of any finite segment of any
orbit is non-degenerate and therefore must be positive definite by a
simple continuity argument, since among all configurations there are
Aubry-Mather configurations having negative definite second
variation.

Let me explain that (a) implies (c). This implication is proved in
\cite {B1} generalizing the limiting argument by Hopf. For any
configuration $\{x_n\}$ with no conjugate points one constructs
stable positive Jacobi field $\{\nu_n\},\ \nu_0=1$ along it. These
$\nu_n$ are measurable positive functions on $\Omega$ which form a
cocycle satisfying discrete Jacobi equation (\ref{Jacobi}).

The implication (c) $\Rightarrow$ (d) is obvious. Just define the
monotone sub-bundle consisting of oriented tangent lines $l(x,\Phi)$
in $T_{(x,\Phi)}\Omega$ spanned by the vectors
$$\frac{\partial}{\partial x}-(L_{11}(x,y)+L_{12}(x,y)\nu_1(x,\Phi)) \frac{\partial}{\partial
\Phi}.$$ It is measurable, monotone, and invariant under $T$. The
invariance follows easily from the discrete Jacobi equation
(\ref{Jacobi}). The last implication we need to prove is
$(d)\Rightarrow(a)$. For this we use the following:
\begin{lemma} Let $\{(x_n,\Phi_n), n=0,..,N\}$ be a finite segment of
an orbit of $T$, $(x_n,\Phi_n)=T^n(x_0,\Phi_0),n=0,..,N$. Let us
assume that $\{w_n, n=0,..,N\}$ is a field of tangent vectors to
$\Omega$ at the points $(x_n,\Phi_n)$ such that $w_n=DT^n(w_0)$ with
the "monotonicity" property, that is $dx(w_n)>0$ for all $0\leq
n\leq N$. Then any Jacobi field along the segment has at most one
"generalized" zero (or sign change, see \cite {E}), in particular
the points $x_i$ and $x_k$ are not conjugate for any $0\leq i<k\leq
N$.
\end{lemma}

\begin{proof}[Proof of Lemma]
Given  $w_n$ like in lemma. Projection of the invariant field $w_n$
defines a \textit{positive} Jacobi field $\{\xi_0,..,\xi_n\}$ along
the configuration $\{x_0,..,x_n\}$. Then it follows from discrete
version of Sturm Separation theorem (\cite{E}, theorem 7.9) that no
Jacobi field can change sign along the segment and in particular no
conjugate points can occur.
\end {proof}

It follows from the lemma that all orbits lying in the support of
the monotone sub-bundle do not have conjugate points. The support is
a set of full measure. Then I claim that all orbits in the closure
of the support (i.e. all the orbits) have no conjugate points as
well. Inded, take any finite segment $\{x_0,.., x_{N+1}\}$ of an
orbit lying in the support of the monotone sub-bundle. Then one can
find a unique Jacobi field $\{1=\xi_0,.., \xi_{N+1}=0\}$. It is
possible since $x_0,x_{N+1}$ are not conjugate. Then it follows from
the lemma $\xi_i, i=1,..,N$ must be positive.

Moreover having a sequence of such segments $\{x^{(k)}_0,..,
x^{(k)}_{N+1}\}$ which converges to a segment $\{y_0,.., y_{N+1}\}$
when $k\rightarrow\infty$ one can pass to the limit and thus to get
a non-negative Jacobi field $\{1=\zeta_0,..,\zeta_N,\zeta_{N+1}=0\}$
along the limiting segment $\{y_0,.., y_{N+1}\}$ with the property
$\zeta_i\geq 0$ . Then it is clear that $\zeta_i$ are in fact
positive for all $i=0,..,N$. So by the lemma the segment $\{y_0,..,
y_{N}\}$ has no conjugate points. Thus all orbits have no conjugate
points and so (a) follows. This completes the proof of the theorem.
\end{proof}

\begin{remark}
It is impossible to drop the monotonicity condition in the
assumption (d) of the last theorem. For example, one can take the
phase portrait of the elliptical billiard in the plane and remove
all the iterates of the zero section of the phase cylinder in order
to get invariant non-vertical sub-bundle defined on a set of full
measure. This sub-bundle is not monotone, and of course there are
conjugate points for elliptical billiard.
\end{remark}

\section{Mirror equation}
We state and prove in this section the mirror equation on a
measurable function $a: \Omega\rightarrow\mathbf{R}$. It follows
from the mirror formula and no conjugate points condition. It could
be proved for general case of billiards in simple Riemannian
domains. But in order to simplify the notations we shall do it only
for constant curvature case. General case is absolutely analogous
with special care taken on the initial points of the Jacobi fields
involved.

\begin{theorem}
If the billiard has no conjugate points, then there exists a
measurable  function on the phase cylinder $a:
\Omega\rightarrow\mathbf{R}$ such that $0<a(x,\Phi)<L(x,\Phi)$ which
satisfies the mirror equation:
\begin{equation} \frac{Y^{'}}{Y}(a(x,\Phi))+\frac{Y^{'}}{Y}(
L(x_{-1},\Phi_{-1})-a(x_{-1},\Phi_{-1}))=\frac{2k(x)}{\sin\phi}
\label{mirror}
\end{equation}

\end{theorem}
For the Euclidean case this theorem was proved in \cite{W} under the
assumption (d) of the existence of invariant monotone sub-bundle of
the billiard map. I shall show how to construct geometrically this
function assuming the condition (a) of no conjugate points. The
geometric meaning of this function is the distance traveled along
the geodesic to the point lying on the caustic. This is defined as
follows.

Let $\{x_n\}$ be a billiard configuration with no conjugate points.
Take the geodesic segment $g_{x_0,x_1}$. Since this geodesic has no
conjugate points there exists and unique orthogonal Jacobi field $J$
along this geodesic which has the following boundary values at the
ends of the segment
$$J(x_0 )=\nu_0 \sin\phi_0 ; \ J(x_1)=-\nu_1(x_0,\Phi_0)
\sin\phi_1.$$ Here $\nu_0=1$ and $\nu_1$ is positive function of the
cocycle $\nu_n$ described in condition (c) of Theorem 3.1, that is
$\nu_0=1$ and $\nu_1$ are the values of the positive Jacobi field at
the points $x_0,x_1$ respectively. Since the signs of $J$ at the
ends of the segment are opposite then there exists a unique point
$z$ between $x_0$ and $x_1$ where $J$ vanishes. So we define
$a(x_0,\Phi_0)$ to be the distance along the segment from $x_0$ to
$z$. One can easily see that $a(x_0,\Phi_0)$ must be the unique
solution of the equation

\begin{equation}
\frac{Y(a)}{Y(L(x_0,\Phi_0)-a)}= \frac{\sin
\phi_0}{\nu_1(x_0,\Phi_0)\sin \phi_1} \label{a0}
\end{equation}

It can be easily seen that the function on the left hand side
defined on $(0,L)$ is strictly monotone increasing (one can compute
the derivative using Wronskian) growing from $0$ at $0$ to $+\infty$
at $L$. Therefore there is a unique solution depending continuously
both on $L$ and on the right hand side. Since the cocycle is
measurable then the function $a$ is measurable also. Analogously to
$a(x_0,\Phi_0)$ we have

\begin{equation}
 \label{a-}
\frac{Y(a(x_{-1},\Phi_{-1}))}{Y(L(x_{-1},\Phi_{-1})-a(x_{-1},\Phi_{-1}))}=
\frac{\sin \phi_{-1}}{\nu_1(x_{-1},\Phi_{-1})\sin\phi_0}=
\end{equation}
$$=\frac{\sin \phi_{-1} \nu_{-1}(x_{0},\Phi_{0})}{\sin\phi_0},$$
where the last equality in (\ref{a-}) follows from the cocycle
property.

The last thing is to show how the mirror equation (\ref{mirror})
follows. It turns out that it is just another form of the Jacobi
equation (\ref{Jacobi1}):
\begin{equation}
 L_{12}(x_{-1},x_0)\nu_{-1}+(L_{22}(x_{-1},x_0)+L_{11}(x_0,x_{1}))
+L_{12}(x_0,x_{1})\nu_{1}=0.\label{Jacobi1}
\end{equation}
Indeed, substituting the explicit formulas for the second
derivatives (\ref{derivatives}) of $L$ and dividing by $\sin^2
\phi_0$ one can rewrite (\ref{Jacobi1}) in the form

\begin{eqnarray}
\label{Jacobi2}
 \lefteqn{\frac{\nu_{-1}(x_{0},\Phi_{0})\sin\phi_{-1}}{Y(L(x_{-1},\Phi_{-1}))\sin
\phi_{0}}+\frac{Y^{'}}{Y}(L(x_{-1},\Phi_{-1}))+}\nonumber\\
   && +\frac{Y^{'}}{Y}(L(x_{0},\Phi_{0}))+
\frac{\nu_{1}(x_{0},\Phi_{0})\sin\phi_{1}}{Y(L(x_{0},\Phi_{0}))\sin
\phi_{0}}=\frac{2k(x_0)}{\sin\phi_0}.
\end{eqnarray}

We can rewrite the first and the last expression in (\ref{Jacobi2})
by (\ref{a-}) and (\ref{a0}) respectively to get

\begin{eqnarray}
\label{Jacobi3}
 \lefteqn{\frac{Y(a(x_{-1},\Phi_{-1}))}{Y(L(x_{-1},\Phi_{-1})-a(x_{-1},\Phi_{-1}))Y(L(x_{-1},\Phi_{-1}))}
 +\frac{Y^{'}}{Y}(L(x_{-1},\Phi_{-1}))}\nonumber\\
   && +\frac{Y^{'}}{Y}(L(x_{0},\Phi_{0}))+
\frac{Y(L(x_{0},\Phi_{0})-a(x_{0},\Phi_{0}))}{Y(L(x_{0},\Phi_{0}))Y(a(x_{0},\Phi_{0}))}=\frac{2k(x_0)}{\sin\phi_0}.
\end{eqnarray}

Now using the fact that Wronskian of any two Jacobi fields is
constant we can write the identities:

$$Y(a)=Y(L)Y^{'}(L-a)-Y^{'}(L)Y(L-a)$$ and
$$Y(L-a)=Y(L)Y^{'}(a)-Y(a)Y^{'}(L),
$$
(in the constant curvature case these are just usual trigonometric
identities). Use these identities on the left hand side of
(\ref{Jacobi3}), the first identity-for the nominator of the first
summand and the second identity-for the nominator of the last
summand. Then it can be easily verified that this yields exactly the
mirror equation (\ref{mirror}) claimed in the theorem.

\section{Proof of rigidity for Hyperbolic plane}
In this section we prove the main theorem for billiards on the
hyperbolic plane.

Write the mirror equation for the Hyperbolic plane (\ref{mirror}):
\begin{equation}
\coth \left(a(x,\Phi)\right)+\coth\left
(L(x_{-1},\Phi_{-1})-a(x_{-1},\Phi_{-1})\right)=\frac{2k(x)}{\sin{\phi}}\label{mhyp}
\end{equation}
Recall that for $t>0$, $\coth t$ is a convex function on $t$  with
$\coth t>1$. Since the equation (\ref{mhyp}) holds true for every
$(x,\Phi)$ we can take $\phi=\pi/2$ and thus immediately get that
$$k(x)>1$$ for any $x$. In other words it follows that our domain must be convex
with respect to horocycles on the Hyperbolic plane. Moreover using
convexity one has
$$
\coth
\frac{a(x,\Phi)+L(x_{-1},\Phi_{-1})-a(x_{-1},\Phi_{-1})}{2}\leq\frac
{k(x)}{\sin{\phi}}.
$$
This can be written in the equivalent form

$$\frac{a(x,\Phi)+L(x_{-1},\Phi_{-1})-a(x_{-1},\Phi_{-1})}{2}\geq
\arctanh\left(\frac{\sin\phi}{k(x)}\right)
$$

Integrate the last inequality with respect to the invariant measure
$d \mu=dx\ d\Phi=\sin\phi\ dx\ d\phi$ to get $$ \int L\ d\mu\geq\
2\int_0^P dx\int_0^{\pi}
\arctanh\left(\frac{\sin\phi}{k(x)}\right)\sin\phi\ d\phi=$$
$$=4\int_0^P dx\int_0^{\pi/2}
\arctanh\left(\frac{\sin\phi}{k(x)}\right)\sin\phi\ d\phi.
$$
Here $P$ is the perimeter of the boundary curve $\gamma$.
 For every $x$ compute the inner integral on the right hand side
 integrating by
parts
$$
\int_0^{\pi/2} \arctanh \left(\frac{\sin\phi}{k(x)}\right)\sin\phi\
d\phi=k(x)\int_0^{\pi/2}\frac{\cos^2\phi}{k^2(x)-\sin^2\phi}\ d\phi=
$$
$$
=\frac{\pi}{2}(k(x)-\sqrt{k^2(x)-1}).
$$
Using Santalo' formula $\int L\ d\mu=2\pi A$ ($A$ is the area of the
domain) we obtain the following inequality
$$
A\geq\int_0^P( k(x)-\sqrt{k^2(x)-1})\ dx.
$$
Use Gauss-Bonnet to write it in the form
$$
A\geq 2\pi+A-\int_0^P\sqrt{k^2(x)-1}\ dx
$$
and therefore
$$
\int_0^P\sqrt{k^2(x)-1}\ dx \geq 2\pi
$$
But then it follows from the next lemma stating the opposite
inequality that the curve $\gamma$ must be a circle.
\begin{lemma}For any simple closed curve $\gamma$ on the Hyperbolic
plane which is convex with respect to horocycles the following
inequality holds true
$$
\int_0^P\sqrt{k^2(x)-1}\ dx \leq 2\pi,
$$
where the equality is possible only for circles.
\end{lemma}
\begin{proof}

The integral can be estimated from above by the Cauchy-Schwartz
$$
\int_0^P\sqrt{k^2(x)-1}\ dx\leq \left(\int_0^P (k(x)-1)\ dx
\right)^{\frac{1}{2}}\left(\int_0^P (k(x)+1)\ dx
\right)^{\frac{1}{2}}=
$$
$$
=((A+2\pi-P)(A+2\pi+P))^{\frac{1}{2}},
$$
where we have applied Gauss Bonnet formula. The last expression
gives
$$((A+2\pi-P)(A+2\pi+P))^{\frac{1}{2}}=(A^2+4\pi A-P^2+4\pi^2)^{\frac{1}{2}}\leq2\pi,$$
since by the isoperimetric inequality on the Hyperbolic plane
$$A^2+4\pi A-P^2\leq 0.$$
\end{proof}
This completes the proof for the Hyperbolic plane.


\section{Proof of rigidity for Hemisphere}
In this section we treat the case of the hemisphere.

Write the mirror equation (\ref{mirror}) for the Hemisphere:
\begin{equation}
\cot \left(a(x,\Phi)\right)+\cot\left
(L(x_{-1},\Phi_{-1})-a(x_{-1},\Phi_{-1})\right)=\frac{2k(x)}{\sin{\phi}}\label{mh}
\end{equation}
We shall need the following easy lemma.
\begin{lemma}
\label{cot}
 The following inequality holds true
$$\cot\left(\frac{a+b}{2}\right)\leq\frac{\cot a+\cot b}{2},$$
for all $a,b$ in the range $(0;\pi)$ satisfying
$\frac{a+b}{2}\leq\pi/2.$
\end{lemma}
\begin{proof}This is obvious in the case when both $a$ and $b$
belong  to $(0;\pi/2]$ just by the convexity of $\cot$ on this
interval. In the remaining case one of the numbers, say $a$ lies in
$(0;\pi/2)$ and $b$ in $(\pi/2;\pi)$. Since the average is
$\leq\pi/2$ we can write $a=\pi/2-x-\delta$ and $b=\pi/2+x$, where
$x,\delta $ are non-negative and $x+\delta<\pi/2$. With these
notations one needs to prove that in this range,
$$\tan \frac{\delta}{2}\leq\frac{\tan(x+\delta)-\tan x}{2}.$$
Indeed we have by the trigonometric formula

$$\frac{\tan(x+\delta)-\tan x}{2}=\frac{(\tan\delta)(1+\tan x
\tan(x+\delta))}{2}\geq\frac{\tan\delta}{2}\geq\tan\frac{\delta}{2}.$$

This proves the lemma.
\end{proof}
Now take the following two numbers
$$a=a(x,\Phi),\ b=L(x_{-1},\Phi_{-1})-a(x_{-1},\Phi_{-1}),
$$
and notice that both of them lie in $(0;\pi)$ since the billiard
domain lies entirely on the hemisphere. Also their average $(a+b)/2$
must be $\leq\pi/2$ since otherwise the sum $\ \cot a+\cot b\ $
would be negative which contradicts the equation (\ref{mh}). Thus
Lemma \ref{cot} can be applied to get:
$$
\cot
\frac{a(x,\Phi)+L(x_{-1},\Phi_{-1})-a(x_{-1},\Phi_{-1})}{2}\leq\frac
{k(x)}{\sin\phi}.
$$
This can be written in the equivalent form:

$$\frac{a(x,\Phi)+L(x_{-1},\Phi_{-1})-a(x_{-1},\Phi_{-1})}{2}\geq
\arctan\left(\frac{\sin\phi}{k(x)}\right).
$$

Integrate the last inequality with respect to the invariant measure
$d \mu=dx\ d\Phi=\sin\phi\ dxd\phi$ to get

$$\int L\ d\mu\ \geq\ 2\int_0^P dx\int_0^{\pi}
\arctan\left(\frac{\sin\phi}{k(x)}\right)\sin\phi\ d\phi=$$
$$=4\int_0^P dx\int_0^{\pi/2}
\arctan\left(\frac{\sin\phi}{k(x)}\right)\sin\phi\ d\phi.
$$
Compute the inner integral on the right hand side
 integrating by
parts
$$
\int_0^{\pi/2} \arctan\left(\frac{\sin\phi}{k(x)}\right)\sin\phi\
d\phi=k(x)\int_0^{\pi/2}\frac{\cos^2\phi}{k^2(x)+\sin^2\phi}\ d\phi=
$$
$$
=\frac{\pi}{2}(\sqrt{k^2(x)+1}-k(x)).
$$
Using Santalo' formula $\int L\ d\mu=2\pi A$ again we obtain the
following inequality
$$
A\geq\int_0^P(\sqrt{k^2(x)+1}-k(x))\ dx.
$$
Use Gauss-Bonnet to write it in the form
$$
A\geq \int_0^P\sqrt{k^2(x)+1}\ dx-(2\pi-A),
$$
which leads to
$$
\int_0^P\sqrt{k^2(x)+1}\ dx\leq 2\pi.
$$
But then the following lemma implies that the curve $\gamma$ must be
a circle thus completing the proof of the main theorem for the
Hemisphere.
\begin{lemma}
For any simple closed curve on the hemisphere the following
inequality holds
$$
\int_0^P\sqrt{k^2(x)+1}\ dx\geq 2\pi,
$$
where the equality happens only for circles.
\end{lemma}

\begin{remark}
Notice that if the sphere is realized as a unite sphere in the
Euclidean 3-space, then lemma follows from Fenchel inequality. This
is because $\sqrt{k^2(x)+1}$ equals to absolute curvature of the
curve in Euclidean space. Next we give an intrinsic proof.
\end{remark}
\begin{proof}
Denote by $I=\int_0^P\sqrt{k^2(x)+1}\ dx$. We have by Cauchy-
Schwartz and Gauss-Bonnet
$$
\int_0^P(\sqrt{k^2(x)+1}+1)\ dx\  \cdot \int_0^P(\sqrt{k^2(x)+1}-1)\
dx\geq\left(\int_0^P k(x)\ dx\right )^2=(2\pi-A)^2.
$$
This can be rewritten as
$$
(I-P)(I+P)\geq (2\pi-A)^2,
$$
and hence
$$I^2\geq P^2+A^2-4\pi A+4\pi ^2\geq 4\pi^2.
$$
Here in the last inequality we used the isoperimetric inequality on
the sphere:
$$
 P^2+A^2-4\pi A \geq0.
$$
Thus $I\geq2\pi$. The proof is completed.
\end{proof}

\end{document}